\newcommand{\powser}[1]{[\![#1]\!]}

\newcommand{\G}{\mathbb{G}}

\newcommand{\Q}{\mathbb{Q}}

\newcommand{\Z}{\mathbb{Z}}

\newcommand{\al}{\alpha}

\newcommand{\lra}[1]{\overset{#1}{\longrightarrow}}

\date{\today}
\documentclass[10pt]{amsart}
\usepackage{amssymb,amsfonts,amsthm,amsmath,verbatim,lscape,color,tensor,url}
\usepackage[margin=1.5in]{geometry}
\usepackage[enableskew]{youngtab}
\usepackage[all]{xy}
\theoremstyle{definition}
\DeclareMathOperator{\Aut}{Aut}

\DeclareMathOperator{\Spf}{Spf}

\DeclareMathOperator{\Sub}{Sub}

\DeclareMathOperator{\Level}{Level}

\DeclareMathOperator{\Tr}{Tr}

\newtheorem{prop}[subsection]{Proposition}
\newtheorem{cor}[subsection]{Corollary}
\newtheorem{lemma}[subsection]{Lemma}

\newtheorem{example}[subsection]{Example}

\newtheorem{remark}[subsection]{Remark}

\makeatletter
\ifx\SK@label\undefined\let\SK@label\label\fi
 \let\your@thm\@thm
 \def\@thm#1#2#3{\gdef\currthmtype{#3}\your@thm{#1}{#2}{#3}}
 \def\mylabel#1{{\let\your@currentlabel\@currentlabel\def\@currentlabel
  {\currthmtype~\your@currentlabel}
 \SK@label{#1@}}\label{#1}}
 
\makeatother

\newcommand\restr[2]{{
  \left.\kern-\nulldelimiterspace 
  #1 
  \vphantom{\big|} 
  \right|_{#2} 
  }}

\DeclareMathOperator{\qz}{\mathbb{T}}

\newcommand{\upperRomannumeral}[1]{\uppercase\expandafter{\romannumeral#1}}

\begin{document}
\title{A canonical lift of Frobenius in Morava $E$-theory}

\author{Nathaniel Stapleton}
\address{Max Planck Institute for Mathematics, Bonn, Germany}
\email{nstapleton@mpim-bonn.mpg.de}

\maketitle

\begin{abstract}
We prove that the $p$th Hecke operator on the Morava $E$-cohomology of a space is congruent to the Frobenius mod $p$. This is a generalization of the fact that the $p$th Adams operation on the complex $K$-theory of a space is congruent to the Frobenius mod $p$. The proof implies that the $p$th Hecke operator may be used to test Rezk's congruence criterion.
\end{abstract}

\section{Introduction}
The $p$th Adams operation on the complex $K$-theory of a space is congruent to the Frobenius mod $p$. This fact plays a role in Adams and Atiyah's proof \cite{adamsatiyahhopf} of the Hopf invariant one problem. It also implies the existence of a canonical operation $\theta$ on $K^0(X)$ satisfying
\[
\psi^p(x) = x^p + p\theta(x),
\]
when $K^0(X)$ is torsion-free. This extra structure was used by Bousfield \cite{bousfieldlambda} to determine the $\lambda$-ring structure of the $K$-theory of an infinite loop space. There are several generalizations of the $p$th Adams operation in complex $K$-theory to Morava $E$-theory: the $p$th additive power operation, the $p$th Adams operation, and the $p$th Hecke operator. In this note, we show that the $p$th Hecke operator is a lift of Frobenius. 
 
In \cite{rezkcongruence}, Rezk studies the relationship between two algebraic structures related to power operations in Morava $E$-theory. One structure is a monad $\qz$ on the category of $E_0$-modules that is closely related to the free $E_{\infty}$-algebra functor. The other structure is a form of the Dyer-Lashof algebra for $E$, called $\Gamma$. Given a $\Gamma$-algebra $R$, each element $\sigma \in \Gamma$ gives rise to a linear endomorphism $Q_{\sigma}$ of $R$. He proves that a $\Gamma$-algebra $R$ admits the structure of an algebra over the monad $\qz$ if and only if there exists an element $\sigma \in \Gamma$ (over a certain element $\bar{\sigma} \in \Gamma/p$) such that $Q_{\sigma}$ is a lift of Frobenius in the following sense:
\[
Q_{\sigma}(r) \equiv r^p \mod pR
\]
for all $r \in R$.

We will show that $Q_{\sigma}$ may be taken to be the $p$th Hecke operator $T_p$ as defined by Ando in \cite[Section 3.6]{Isogenies}. We prove this by producing a canonical element $\sigma_{can} \in \Gamma$ lifting the Frobenius class $\bar{\sigma} \in \Gamma/p$ \cite[Section 10.3]{rezkcongruence} such that $Q_{\sigma_{can}} = T_p$. This provides us with extra algebraic structure on torsion-free algebras over the monad $\qz$ in the form of a canonical operation $\theta$ satisfying
\[
T_p(r) = r^p + p\theta(r).
\]

Let $\G_{E_0}$ be the formal group associated to $E$, a Morava $E$-theory spectrum. The Frobenius $\phi$ on $E_0/p$ induces the relative Frobenius isogeny
\[
\G_{E_0/p} \lra{} \phi^*\G_{E_0/p}
\]
over $E_0/p$. The kernel of this isogeny is a subgroup scheme of order $p$. By a theorem of Strickland, this corresponds to an $E_0$-algebra map
\[
\bar{\sigma} \colon E^0(B\Sigma_p)/I \lra{} E_0/p,
\]
where $I$ is the image of the transfer from the trivial group to $\Sigma_p$. This map further corresponds to an element in the mod $p$ Dyer-Lashof algebra $\Gamma/p$. Rezk considers the set of $E_0$-module maps $[\bar{\sigma}] \subset \hom(E^0(B\Sigma_p)/I,E_0)$ lifting $\bar{\sigma}$.

\begin{prop}
There is a canonical choice of lift $\sigma_{can} \in [\bar{\sigma}]$.
\end{prop}

The construction of $\sigma_{can}$ is an application of the formula for the $K(n)$-local transfer (induction) along the surjection from $\Sigma_p$ to the trivial group \cite[Section 7.3]{Ganterexponential}. 

Let $X$ be a space and let
\[
P_p/I \colon  E^0(X) \lra{} E^0(B\Sigma_p)/I \otimes_{E_0} E^0(X)
\]
be the $p$th additive power operation. The endomorphism $Q_{\sigma_{can}}$ of $E^0(X)$ is the composite of $P_p/I$ with $\sigma_{can} \otimes 1$. 
\begin{prop}
For any space $X$, the following operations on $E^0(X)$ are equal: 
\[
Q_{\sigma_{can}} = (\sigma_{can} \otimes 1)(P_p/I) = T_p.
\]
\end{prop}
This has the following immediate consequence:
\begin{cor}
Let $X$ be a space such that $E^0(X)$ is torsion-free. There exists a canonical operation
\[
\theta \colon E^0(X) \lra{} E^0(X)
\]
such that, for all $x \in E^0(X)$,  
\[
T_p(x) = x^p + p\theta(x).
\]
\end{cor}

\emph{Acknowledgements} It is a pleasure to thank Tobias Barthel, Charles Rezk, Tomer Schlank, and Mahmoud Zeinalian for helpful discussions and to thank the Max Planck Institute for Mathematics for its hospitality.

\section{Tools} \label{tools}
Let $E$ be a height $n$ Morava $E$-theory spectrum at the prime $p$. We will make use of several tools that let us access $E$-cohomology. We summarize them in this section. 

For the remainder of this paper, let $E(X) = E^0(X)$ for any space $X$. We will also write $E$ for the coefficients $E^0$ unless we state otherwise.

\emph{Character theory} Hopkins, Kuhn, and Ravenel introduce character theory for $E(BG)$ in \cite{hkr}. They construct the rationalized Drinfeld ring $C_0$ and introduce a ring of generalized class functions taking values in $C_0$:
\[
Cl_n(G,C_0) = \{\text{$C_0$-valued functions on conjugacy classes of map from $\Z_{p}^n$ to $G$}\}.
\] 
They construct a map
\[
E(BG) \lra{} Cl_n(G,C_0)
\]
and show that it induces an isomorphism after the domain has been base-changed to $C_0$ \cite[Theorem C]{hkr}. When $n=1$, this is a $p$-adic version of the classical character map from representation theory.

\emph{Good groups} A finite group $G$ is good if the character map
\[
E(BG) \lra{} Cl_n(G,C_0)
\]
is injective. Hopkins, Kuhn, and Ravenel show that $\Sigma_{p^k}$ is good for all $k$ \cite[Theorem 7.3]{hkr}. 

\emph{Transfer maps} It follows from a result of Greenlees and Sadofsky \cite{greenlees-sadofsky} that there are transfer maps in $E$-cohomology along all maps of finite groups. In \cite[Section 7.3]{Ganterexponential}, Ganter studies the case of the transfer from $G$ to the trivial group and shows that there is a simple formula for the transfer on the level of class functions. Let 
\[
\Tr_{C_0} \colon Cl_n(G,C_0) \lra{} C_0
\]
be given by the formula $f \mapsto \frac{1}{|G|}\sum_{[\al]}f([\al])$, where the sum runs over conjugacy classes of maps $\al \colon \Z_{p}^n \rightarrow G$. Ganter shows that there is a commutative diagram
\[
\xymatrix{E(BG) \ar[r]^-{\Tr_{E}} \ar[d] & E \ar[d] \\ Cl_n(G) \ar[r]^-{\Tr_{C_0}} & C_0,}
\]
in which the vertical maps are the character map.

\emph{Subgroups of formal groups} Let $\G_{E} = \Spf(E(BS^1))$ be the formal group associated to the spectrum $E$. In \cite{etheorysym}, Strickland produces a canonical isomorphism
\[
\Spf(E(B\Sigma_{p^k})/I) \cong \Sub_{p^k}(\G_{E}),
\]
where $I$ is the image of the transfer along $\Sigma_{p^{k-1}}^{\times p} \subset \Sigma_{p^k}$ and $\Sub_{p^k}(\G_{E})$ is the scheme that classifies subgroup schemes of order $p^k$ in $\G_E$. We will only need the case $k=1$.

\emph{The Frobenius class} The relative Frobenius is a degree $p$ isogeny of formal groups
\[
\G_{E/p} \lra{} \phi^*\G_{E/p},
\]
where $\phi \colon E/p \rightarrow E/p$ is the Frobenius. The kernel of the map is a subgroup scheme of order $p$. Using Strickland's result, there is a canonical map of $E$-algebras
\[
\bar{\sigma} \colon E(B\Sigma_p)/I \lra{} E/p
\]
picking out the kernel. In \cite[Section 10.3]{rezkcongruence}, Rezk describes this map in terms of a coordinate and considers the set of $E$-module maps $[\bar{\sigma}] \subset \hom(E(B\Sigma_p),E)$ that lift $\bar{\sigma}$.

\emph{Power operations} In \cite{structuredmoravae}, Goerss, Hopkins, and Miller prove that the spectrum $E$ admits the structure of an $E_{\infty}$-ring spectrum in an essentially unique way. This implies a theory of power operations. These are natural multiplicative non-additive maps
\[
P_m  \colon  E(X) \lra{} E(B\Sigma_{m}) \otimes_E E(X)
\]
for all $m>0$. For $m=p^k$, they can be simplified to obtain interesting ring maps by further passing to the quotient
\[
P_{p^k}/I \colon E(X) \lra{} E(B\Sigma_{p^k}) \otimes_E E(X) \lra{} E(B\Sigma_{p^k})/I \otimes_E E(X),
\]
where $I$ is the transfer ideal that appeared above.

\emph{Hecke operators} In \cite[Section 3.6]{Isogenies}, Ando produces operations 
\[
T_{p^k} \colon E(X) \lra{} E(X)
\]
by combining the structure of power operations, Strickland's result, and ideas from character theory. Let $\qz = (\Q_p/\Z_p)^n$, let $H \subset \qz$ be a finite subgroup, and let $D_{\infty}$ be the Drinfeld ring at infinite level so that $\Spf(D_{\infty}) = \Level(\qz,\G_{E})$ and $\Q \otimes D_{\infty} = C_0$. Ando constructs an Adams operation depending on $H$ as the composite
\[
\psi^H \colon E(X) \lra{P_p/I} E(B\Sigma_p)/I \otimes_E E(X) \lra{H\otimes 1} D_{\infty} \otimes_E E(X).
\] 
He then defines the $p^k$th Hecke operator 
\[
T_{p^k} = \sum_{\substack{H \subset \qz \\ |H| = p^k}} \psi^H
\]
and shows that this lands in $E(X)$. 

\section{A canonical representative of the Frobenius class}
We construct a canonical representative of the set $[\bar{\sigma}]$. The construction is an elementary application of several of the tools presented in the previous section.

We specialize the transfers of the previous section to $G = \Sigma_p$. Let 
\[
\Tr_E \colon E(B\Sigma_p) \lra{} E
\]
be the transfer from $\Sigma_p$ to the trivial group and let
\[
\Tr_{C_0} \colon Cl_n(\Sigma_p, C_0) \lra{} C_0
\]
be the transfer in class functions from $\Sigma_p$ to the trivial group. This is given by the formula
\[
\Tr_{C_0}(f) = \frac{1}{p!}\sum_{[\al]}f([\al]).
\]
Recall that $\qz = (\Q_p/\Z_p)^n$ and let $\Sub_p(\qz)$ be the set of subgroups of order $p$ in $\qz$.

\begin{lemma} \label{sigmap} \cite[Section 4.3.6]{marshthesis}
The restriction map along $\Z/p \subseteq \Sigma_p$ induces an isomorphism
\[
E(B\Sigma_p) \lra{\cong} E(B\Z/p)^{\Aut(\Z/p)}.
\]
After a choice of coordinate $x$,
\[
E(B\Sigma_p) \cong E[y]/(yf(y)),
\]
where the degree of $f(y)$ is 
\[
|\Sub_p(\qz)| = \frac{p^n-1}{p-1} = \sum_{i=0}^{n-1}p^i, 
\]
$f(0)=p$, and $y$ maps to $x^{p-1}$ in $E(B\Z/p) \cong E\powser{x}/[p](x)$.
\end{lemma}


\begin{lemma} \cite[Proposition 4.2]{Quillenelementary}
After choosing a coordinate, there is an isomorphism 
\[
E(B\Sigma_p)/I \cong E[y]/(f(y)),
\]
and the ring is free of rank $|\Sub_p(\qz)|$ as an $E$-module.
\end{lemma}

After choosing a coordinate, the restriction map $E(B\Sigma_p) \rightarrow E$ sends $y$ to $0$ and the map
\[
E(B\Sigma_p) \rightarrow E(B\Sigma_p)/I
\]
is the quotient by the ideal generated by $f(y)$.

\begin{lemma} \label{index}
The index of the $E$-module $E(B\Sigma_p)$ inside $E \times E(B\Sigma_p)/I$ is $p$.
\end{lemma}
\begin{proof}
This can be seen using the coordinate. There is a basis of $E(B\Sigma_p)$ given by the set $\{1, y, \ldots, y^m\}$, where $m = |\Sub_p(\qz)|$, and a basis of $E \times E(B\Sigma_p)/I$ given by 
\[
\{(1,0),(0,1),(0,y), \ldots, (0,y^{m-1})\}.
\] 
By Lemma \ref{sigmap}, the image of the elements $\{1, y, \ldots, y^{m-1}, p-f(y)\}$ in $E(B\Sigma_p)$ is the set
\[
\{(1,1),(0,y), \ldots, (0,y^{m-1}), (0,p)\}
\]
in $E \times E(B\Sigma_p)/I$. The image of $y^m$ is in the span of these elements and the submodule generated by these elements has index $p$.
\end{proof}

\begin{lemma} \cite[Section 10.3]{rezkcongruence}
In terms of a coordinate, the Frobenius class
\[
\bar{\sigma} \colon E(B\Sigma_p)/I \lra{} E/p
\] 
is the quotient by the ideal $(y)$.
\end{lemma}

Now we modify $\Tr_{C_0}$ to construct a map
\[
\sigma_{can} \colon E(B\Sigma_p)/I \lra{} E.
\]
By Ganter's result \cite[Section 7.3]{Ganterexponential} and the fact that $\Sigma_p$ is good, the restriction of $\Tr_{C_0}$ to $E(B\Sigma_{p})$ is equal to $\Tr_{E}$. It makes sense to restrict $\Tr_{C_0}$ to 
\[
E \times E(B\Sigma_p)/I \subset Cl_n(\Sigma_p,C_0).
\]
Lemma \ref{index} implies that this lands in $\frac{1}{p}E$. Thus we see that the target of the map
\[
\restr{p!\Tr_{C_0}}{E \times E(B\Sigma_p)/I}
\]
can be taken to be $E$. We may further restrict this map to the subring $E(B\Sigma_p)/I$ to get
\[
\restr{p!\Tr_{C_0}}{E(B\Sigma_p)/I} \colon E(B\Sigma_p)/I \lra{} E.
\]
From the formula for $\Tr_{C_0}$, for $e \in E \subset E(B\Sigma_p)/I$, we have
\[
\restr{p!\Tr_{C_0}}{E(B\Sigma_p)/I}(e) = |\Sub_p(\qz)|e.
\]
Note that $|\Sub_p(\qz)|$ is congruent to $1$ mod $p$ (and therefore a $p$-adic unit). We set
\[
\sigma_{can} = \restr{p!\Tr_{C_0}}{E(B\Sigma_p)/I}.
\]

\begin{remark}
One may also normalize $\sigma_{can}$ by dividing by $|\Sub_p(\qz)|$ so that $e$ is sent to $e$.
\end{remark}


We now show that $\sigma_{can}$ fits in the diagram
\[
\xymatrix{& E \ar[d] \\ E(B\Sigma_p)/I \ar[r]_-{\bar{\sigma}} \ar[ru]^-{\sigma_{can}} & E/p,}
\]
where $\bar{\sigma}$ picks out the kernel of the relative Frobenius.

\begin{prop}
The map
\[
\sigma_{can} \colon E(B\Sigma_p)/I \lra{} E
\]
is a representative of Rezk's Frobenius class.
\end{prop}
\begin{proof}
We may be explicit. Choose a coordinate so that the quotient map
\[
q \colon E(B\Sigma_p) \lra{} E(B\Sigma_p)/I
\] 
is given by
\[
q \colon E[y]/(yf(y)) \lra{} E[y]/(f(y)).
\]
We must show that 
\[
\xymatrix{E(B\Sigma_p)/I \ar[r]^-{\sigma_{can}} & E \ar[r]^-{\text{mod } p} & E/p}
\]
is the quotient by the ideal $(y) \subset E(B\Sigma_p)/I$.

There is a basis of $E(B\Sigma_p)$ (as an $E$-module) given by $\{1,y,\ldots,y^{m}\}$, where $m = |\Sub_p(\qz)|$. We will be careful to refer to the image of $y^i$ in $E(B\Sigma_p)/I$ as $q(y^i)$. For the basis elements of the form $y^i$, where $i \neq 0$, the restriction map $E(B\Sigma_p) \rightarrow E$ sends $y^i$ to $0$. Thus
\[
\Tr_{E}(y^i) = \restr{\Tr_{C_0}}{E(B\Sigma_p)/I}(q(y^i)) \in E.
\]
Now the definition of $\sigma_{can}$ implies that $\sigma_{can}(q(y^i))$ is divisible by $p$. So 
\[
\sigma_{can}(q(y^i)) \equiv 0 \mod p.
\]
It is left to show that, for $e$ in the image of $E \rightarrow E(B\Sigma_p)/I$, 
\[
\sigma_{can}(e) \equiv e \mod p.
\]
We have already seen that
\[
\restr{p!\Tr_{C_0}}{E(B\Sigma_p)/I}(e) = |\Sub_p(\qz)|e.
\]
The result follows from the fact that $|\Sub_p(\qz)| \equiv 1$ mod $p$.
\end{proof}

\section{The Hecke operator congruence}
We show that the $p$th additive power operation composed with $\sigma_{can}$ is the $p$th Hecke operator. This implies that the Hecke operator satisfies a certain congruence.

The two maps in question are the composite
\[
\xymatrix{E(X) \ar[r]^-{P_p/I} & E(B\Sigma_p)/I \otimes_E E(X) \ar[r]^-{\sigma_{can} \otimes 1}& E(X)}
\]
and the Hecke operator $T_p$ described in Section \ref{tools}. 

\begin{prop}
The $p$th additive power operation composed with the canonical representative of the Frobenius class is equal to the $p$th Hecke operator: 
\[
(\sigma_{can} \otimes 1)(P_{p}/I) = T_p.
\]
\end{prop}
\begin{proof}
This follows in a straight-forward way from the definitions. Unwrapping the definition of the character map, the map $\sigma_{can}$ is the sum of a collection of maps 
\[
E(B\Sigma_p)/I \lra{} C_0,
\] 
one for each subgroup of order $p$ in $\qz$. These are the maps induced by the canonical isomorphism
\[
C_0 \otimes \Sub_p(\G_{E}) \cong \Sub_p(\qz).
\]
In other words, they classify the subgroups of order $p$ in $\qz$.
\end{proof}

Since $\sigma_{can} \in [\bar{\sigma}]$, the following diagram commutes
\[
\xymatrix{E(X) \ar[r]^-{P_p} & E(B\Sigma_p) \otimes_E E(X) \ar[r] \ar[d]_-{\text{Res}\otimes 1} & E(B\Sigma_p)/I \otimes_E E(X) \ar[d]_-{\bar{\sigma}\otimes 1} \ar[r]^-{\sigma_{can} \otimes 1} & E(X) \ar[dl] \\ & E(X) \ar[r] & E(X)/p &}
\]
and this implies that 
\[
(\sigma_{can}\otimes 1)(P_p/I)(x) \equiv x^p \mod p.
\]

\begin{cor}
For $x \in E(X)$, there is a congruence
\[
T_p(x) \equiv x^p \mod p.
\]
\end{cor}

Let $X$ be a space with the property that $E(X)$ is torsion-free.  The corollary above implies the existence of a canonical function
\[
\theta \colon E(X) \lra{} E(X)
\] 
such that 
\[
T_p(x) = x^p + p\theta(x).
\]

\begin{example} 
When $n=1$, $\G_E$ is a height $1$ formal group,
\[
E(B\Sigma_p)/I
\]
is a rank one $E$-module, and $\sigma_{can}$ is an $E$-algebra isomorphism. The composite
\[
\xymatrix{E(X) \ar[r]^-{P_p/I} & E(B\Sigma_p)/I \otimes_E E(X) \ar[r]^-{\sigma_{can} \otimes 1} & E(X)}
\]
is the $p$th unstable Adams operation. In this situation, the function $\theta$ is understood by work of Bousfield \cite{bousfieldlambda}.
\end{example}

\begin{example}
At arbitrary height, we may consider the effect of $T_p$ on $z \in \Z_p \subset E$. Since $T_p$ is a sum of ring maps 
\[
T_p(z) = |\Sub_p(\qz)|z.
\]
This is congruent to $z^p$ mod $p$.
\end{example}

\begin{example}
At height $2$ and the prime $2$, Rezk constructed an $E$-theory associated to a certain elliptic curve \cite{rezkpowercalc}. He calculated $P_2/I$, when $X=*$. He found that, after choosing a particular coordinate $x$, 
\[
E(B\Sigma_2)/I \cong \Z_2\powser{u_1}[x]/(x^3-u_1x-2)
\]
and
\[
P_2/I \colon \Z_2\powser{u_1} \lra{} \Z_2\powser{u_1}[x]/(x^3-u_1x-2)
\]
sends $u_1 \mapsto u_{1}^2+3x-u_1x^2$. In \cite[Section 4B]{Drinfeld}, Drinfeld explains how to compute the ring that corepresents $\Z/2\times \Z/2$-level structures. Note that in the ring
\[
\Z_2\powser{u_1}[y,z]/(y^3-u_1y-2),
\]
$y$ is a root of $z^3-u_1z-2$ and
\[
\frac{z^3-u_1z-2}{z-y} = z^2+yz+y^2-u_1.
\]
Drinfeld's construction gives
\[
D_1 = \Gamma \Level(\Z/2 \times \Z/2, \G_{E}) \cong \Z_2\powser{u_1}[y,z]/(y^3-u_1y-2,z^2+yz+y^2-u_1).
\]
The point of this construction is that $x^3-u_1x-2$ factors into linear terms over this ring. In fact,
\[
x^3-u_1x-2 = (x-y)(x-z)(x+y+z).
\]
The three maps $E(B\Sigma_2)/I \rightarrow D_1 \subset C_0$ that show up in the character map are given by sending $x$ to these roots. A calculation shows that 
\[
\sigma_{can}(x) = 0
\]
and that 
\[
T_p(u_1) = (\sigma_{can} \otimes 1)(P_2/I)(u_1) = u_{1}^2.
\]
\end{example}

\bibliographystyle{amsalpha}
\bibliography{mybib}

\end{document}